\documentclass{amsart}
\usepackage{latexsym, amssymb, amsmath, amsthm}
\usepackage[T1]{fontenc}
\usepackage{lmodern}
\usepackage{enumerate}
\usepackage{mathabx}
\usepackage{csquotes}
\usepackage[mathscr]{euscript}

\usepackage{tikz}
\usetikzlibrary{trees}

\usepackage{pifont}

\newtheorem{theorem}{Theorem}[section]
\newtheorem{lemma}[theorem]{Lemma}
\newtheorem{corollary}[theorem]{Corollary}
\newtheorem{proposition}[theorem]{Proposition}
\newtheorem{question}[theorem]{Question}

\newtheorem{claim}[theorem]{Claim}

\theoremstyle{definition}
\newtheorem{definition}[theorem]{Definition}
\theoremstyle{remark}
\newtheorem{remark}[theorem]{Remark}

\newbox\gnBoxA
\newdimen\gnCornerHgt
\setbox\gnBoxA=\hbox{$\ulcorner$}
\global\gnCornerHgt=\ht\gnBoxA
\newdimen\gnArgHgt
\def\gnmb #1{%
\setbox\gnBoxA=\hbox{$#1$}%
\gnArgHgt=\ht\gnBoxA%
\ifnum     \gnArgHgt<\gnCornerHgt \gnArgHgt=0pt%
\else \advance \gnArgHgt by -\gnCornerHgt%
\fi \raise\gnArgHgt\hbox{$\ulcorner$} \box\gnBoxA %
\raise\gnArgHgt\hbox{$\urcorner$}}

\title{On the inevitability of the consistency operator}

\author{James Walsh}
\address{Group in Logic and the Methodology of Science, University of California, Berkeley}
\email{walsh@math.berkeley.edu}
\author{Antonio Montalb{\'a}n}
\address{Department of Mathematics, University of California, Berkeley}
\email{antonio@math.berkeley.edu}


\begin{document}

\maketitle

\begin{abstract}
We examine recursive monotonic functions on the Lindenbaum algebra of $\mathsf{EA}$. We prove that no such function sends every consistent $\varphi$ to a sentence with deductive strength strictly between $\varphi$ and $(\varphi\wedge\mathsf{Con}(\varphi))$. We generalize this result to iterates of consistency into the effective transfinite. We then prove that for any recursive monotonic function $f$, if there is an iterate of $\mathsf{Con}$ that bounds $f$ everywhere, then $f$ must be somewhere equal to an iterate of $\mathsf{Con}$.
\end{abstract}

\section{Introduction}

It is a well-known empirical phenomenon that natural axiomatic theories are well-ordered by their consistency strength. However, without a precise mathematical definition of ``natural,'' it is difficult to explain this observation in a strictly mathematical way. One expression of this phenomenon comes from \emph{ordinal analysis}, a research program whereby recursive ordinals are assigned to theories as a measurement of their consistency strength. One method for calculating the proof-theoretic ordinal of a theory $T$ involves demonstrating that $T$ can be approximated over a weak base theory by a class of formulas that are well understood. In particular, the $\Pi^0_1$ fragments of natural theories are often proof-theoretically equivalent to iterated consistency statements over a weak base theory, making these theories amenable to ordinal analysis. For discussion, see, e.g., Beklemishev \cite{beklemishev2003proof,beklemishev2004provability} and Joosten \cite{joosten2016turing}. 

Why are the $\Pi^0_1$ fragments of natural theories proof-theoretically equivalent to iterated consistency statements? Our approach to this question is inspired by Martin's approach to another famous question from mathematical logic: why are natural Turing degrees well-ordered by Turing reducibility? Martin conjectured that (i) the non-constant degree invariant functions meeting a certain simplicity condition ($f\in L(\mathbb{R})$)\footnote{Martin's Conjecture is stated under the hypothesis $\mathsf{ZF}+\mathsf{AD}+\mathsf{DC}$, which is satisfied by $L(\mathbb{R})$ assuming that there are $\omega$ many Woodin cardinals with a measurable above them all.} are pre-well-ordered by the relation ``$f(a)\leq_Tg(a)$ on a cone in the Turing degrees'' and (ii) the successor for this well-ordering is induced by the Turing jump. Martin's conjecture is meant to capture the idea that the Turing jump and its iterates into the transfinite are the only natural non-trivial degree invariant functions.

In this paper we investigate analogous hypotheses concerning jumps on consistent axiomatic theories, namely, consistency statements. We fix elementary arithmetic $\mathsf{EA}$ as our base theory. $\mathsf{EA}$ is a subsystem of $\mathsf{PA}$ that is often used as a base theory in ordinal analysis and in which standard approaches to arithmetization of syntax can be carried out without substantial changes; see \cite{beklemishev2005reflection} for details. We write $[\varphi]$ to denote the equivalence class of $\varphi$ modulo $\mathsf{EA}$-provable equivalence. We write $\varphi\vdash\psi$ if $\mathsf{EA}\vdash\varphi\rightarrow\psi$ and say that $\varphi$ \emph{implies} $\psi$. If $\varphi\vdash\psi$ but $\psi\nvdash\varphi$ we say that $\varphi$ \emph{strictly implies} $\psi$. The \emph{Lindenbaum algebra} of $\mathsf{EA}$ is the set of equivalence classes of sentences ordered by $\vdash$. We focus on recursive functions $f$ that are \emph{monotonic}, i.e.,
$$\textrm{if } \varphi\vdash\psi \textrm{, then }f(\varphi)\vdash f(\psi).$$ 
We note that (i) a function $f$ is monotonic just in case $f$ preserves implication over $\mathsf{EA}$ and (ii) all monotonic functions induce functions on the Lindenbaum algebra of $\mathsf{EA}$. We adopt the convention that all functions named ``$f$'' in this paper are recursive.

Our goal is to demonstrate that $\varphi\mapsto(\varphi\wedge \mathsf{Con}(\varphi))$ and its iterates into the transfinite are canonical among monotonic functions. Our first theorem to this end is the following.

\begin{theorem}
\label{aaa}
Let $f$ be monotonic. Suppose that for all consistent $\varphi$,\\
(i) $\varphi\wedge \mathsf{Con}(\varphi)$ implies $f(\varphi)$ and\\
(ii) $f(\varphi)$ strictly implies $\varphi$. \\
Then for every true $\varphi$, there is a true $\psi$ such that $\psi\vdash\varphi$ and $[f(\psi)] = [\psi\wedge \mathsf{Con}(\psi)].$
\end{theorem}

\begin{corollary} \label{no push down}
There is no monotonic function $f$ such that for all consistent $\varphi$, \\
(i) $\varphi\wedge\mathsf{Con}(\varphi)$ strictly implies $f(\varphi)$ and \\
(ii) $f(\varphi)$ strictly implies $\varphi$.
\end{corollary}

We note that this result depends essentially on the condition of monotonicity. Shavrukov and Visser \cite{shavrukov2014uniform} studied recursive functions $f$ that are \emph{extensional} over the Lindenbaum algebra of $\mathsf{PA}$, i.e., 
$$\textrm{if }\mathsf{PA}\vdash(\varphi\leftrightarrow\psi) \textrm{, then } \mathsf{PA}\vdash(f(\varphi)\leftrightarrow f(\psi)),$$ 
and proved the following theorem. 
\begin{theorem}
\label{shavrukov visser}
(Shavrukov--Visser)
There is a recursive extensional function $f$ such that for all consistent $\varphi$, \\
(i) $\varphi\wedge\mathsf{Con}(\varphi)$ strictly implies $f(\varphi)$ and \\
(ii) $f(\varphi)$ strictly implies $\varphi$.
\end{theorem}
In particular, Shavrukov and Visser proved that for any consistent $\varphi$, the sentence
$$\varphi^\star:=\varphi \wedge \forall x\big{(}\mathsf{Con}(I\Sigma_x+\varphi)\rightarrow\mathsf{Con}(I\Sigma_x+\varphi+\mathsf{Con}(I\Sigma_x+\varphi))\big{)}$$
has deductive strength strictly between $\varphi$ and $\varphi\wedge\mathsf{Con}(\varphi)$, and that the map $\varphi\mapsto\varphi^\star$ is extensional. By a theorem of Kripke and Pour-El \cite{pour1967deduction}, the Lindenbaum algebras of $\mathsf{PA}$ and $\mathsf{EA}$ are effectively isomorphic, whence Theorem \ref{shavrukov visser} also applies to $\mathsf{EA}$. Thus, Corollary \ref{no push down} cannot be strengthened by weakening the hypothesis of monotonicity to the hypothesis of extensionality.

We also note that Friedman, Rathjen, and Weiermann \cite{friedman2013slow} introduced a notion of \emph{slow consistency} with which they produced a $\Pi^0_1$ sentence $\mathsf{SlowCon}(\mathsf{PA})$ with deductive strength strictly between $\mathsf{PA}$ and $\mathsf{PA}+\mathsf{Con}(\mathsf{PA})$. In general, the statement $\mathsf{SlowCon}(\varphi)$ has the form
$$\forall x (F_{\epsilon_0}(x)\downarrow \rightarrow \mathsf{Con}(I\Sigma_x+\varphi))$$
where $F_{\epsilon_0}$ is a standard representation of a recursive function that is not provably total in $\mathsf{PA}$. This is not in conflict with Corollary \ref{no push down}, however, since $\varphi\wedge\mathsf{Con}(\varphi)$ and $\varphi\wedge\mathsf{SlowCon}(\varphi)$ are provably equivalent for all $\varphi$ such that $\varphi\vdash \forall x F_{\epsilon_0}(x)\downarrow$. On the other hand, changing the definition of the $\mathsf{SlowCon}(\varphi)$ so that the function in the antecedent varies with the input $\varphi$ results in a map that is not monotonic.

Theorem \ref{aaa} generalizes to the iterates of $\mathsf{Con}$ into the effective transfinite. For an elementary presentation $\alpha$ of a recursive well-ordering (see Definition \ref{elementary}) and a sentence $\varphi$, we define sentences $\mathsf{Con}^\beta(\varphi)$ for every $\beta<\alpha$.
\begin{flalign*}
\mathsf{Con}^0(\varphi):=&\top\\
\mathsf{Con}^{\beta+1}(\varphi):=&\mathsf{Con}(\varphi\wedge \mathsf{Con}^\beta(\varphi))\\
\mathsf{Con}^\lambda(\varphi):=&\forall\beta<\lambda(\mathsf{Con}^\beta(\varphi))
\end{flalign*}
For a precise definition using G\"{o}del's fixed point lemma, see Definition \ref{iterates}. Note that for every $\varphi$, $[\mathsf{Con}^1(\varphi)]=[\mathsf{Con}(\varphi)]$. 

\begin{remark}
We warn the reader that there is some discrepancy between our notation and the notation used by other authors. Our iteration scheme $\mathsf{Con}^{\alpha+1}(\varphi)\equiv\mathsf{Con}(\varphi\wedge\mathsf{Con}^\alpha(\varphi))$ is sometimes denoted $\mathsf{Con}((\mathsf{EA}+\varphi)_\alpha)$, e.g., \cite{beklemishev1995iterated}. Moreoever, the notation $\mathsf{Con}^{\alpha+1}(\varphi)$ is sometimes used to denote $\mathsf{Con}(\mathsf{Con}^\alpha(\varphi))$, e.g., \cite{beklemishev1991provability}.
\end{remark}

With each predicate $\mathsf{Con}^\alpha$ we associate a function 
$$\varphi\mapsto(\varphi\wedge \mathsf{Con}^\alpha(\varphi)).$$ 

Theorem \ref{aaa} then generalizes into the effective transfinite as follows.

\begin{theorem}\label{generalization of aaa}
Let $f$ be monotonic. Suppose that for all $\varphi$, \\
(i) $\varphi\wedge \mathsf{Con}^\alpha(\varphi)$ implies $f(\varphi)$,\\
(ii) if $[f(\varphi)]\neq[\bot]$, then $f(\varphi)$ strictly implies $\varphi\wedge \mathsf{Con}^\beta(\varphi)$ for all $\beta<\alpha.$ \\
Then for every true $\varphi$, there is a true $\psi$ such that $\psi\vdash\varphi$ and $[f(\psi)] = [\psi\wedge \mathsf{Con}^\alpha(\psi)].$
\end{theorem}

\begin{corollary}\label{generalization of no push down}
There is no monotonic $f$ such that for all $\varphi$, if $[\varphi\wedge\mathsf{Con}^\alpha(\varphi)]\neq[\bot]$, then both\\
(i) $\varphi\wedge \mathsf{Con}^\alpha(\varphi)$ strictly implies $f(\varphi)$ and\\
(ii) $f(\varphi)$ strictly implies $\varphi\wedge \mathsf{Con}^\beta(\varphi)$ for all $\beta<\alpha$.
\end{corollary}

Thus, if the range of a monotonic function $f$ is sufficiently constrained, then for some $\varphi$ and some $\alpha$, 
$$[f(\varphi)]=[\varphi\wedge \mathsf{Con}^\alpha(\varphi)]\neq[\bot].$$
This property still holds even when these constraints on the range of $f$ are relaxed considerably. More precisely, if a monotonic function is everywhere bounded by a finite iterate of $\mathsf{Con}$, then it must be somewhere equivalent to an iterate of $\mathsf{Con}$.

\begin{theorem}\label{finite version of zzz}
Let $n\in\mathbb{N}$. Let $f$ be a monotonic function such that for every $\varphi$, \\
(i) $\varphi\wedge \mathsf{Con}^n(\varphi)$ implies $f(\varphi)$ and \\
(ii) $f(\varphi)$ implies $\varphi$. \\
Then for some $\varphi$ and some $k\leq n$, $[f(\varphi)]=[\varphi\wedge \mathsf{Con}^k(\varphi)]\neq[\bot].$
\end{theorem}

To generalize this result into the effective transfinite, we focus on a particular class of monotonic functions that we call $\Pi^0_1$. 

\begin{definition}
A function $f$ is \emph{$\Pi^0_1$} if $f(\varphi)\in\Pi^0_1$ for all $\varphi$.
\end{definition}

Our main theorem is the following: if a monotonic function is everywhere bounded by a transfinite iterate of $\mathsf{Con}$, then it must be somewhere equivalent to an iterate of $\mathsf{Con}$. This to say that the iterates of the consistency operator are \emph{inevitable}; no monotonic function that is everywhere bounded by some iterate of $\mathsf{Con}$ can avoid all of the iterates of $\mathsf{Con}$.

\begin{theorem}
\label{zzz}
Let $\varphi\mapsto f(\varphi)$ be a monotonic $\Pi^0_1$ function Then either \\
(i) for some $\beta\leq\alpha$ and some $\varphi$, $[\varphi\wedge f(\varphi)]=[\varphi\wedge \mathsf{Con}^\beta(\varphi)]\neq[\bot]$ or\\
(ii) for some $\varphi$, $(\varphi\wedge\mathsf{Con}^\alpha(\varphi))\nvdash f(\varphi)$.
\end{theorem}

The main theorem bears a striking similarity to the following theorem of Slaman and Steel \cite{slaman1988definable}.

\begin{theorem}
\label{Slaman-Steel}
(Slaman--Steel) Suppose $f: 2^\omega\rightarrow2^\omega$ is Borel, order-preserving with respect to $\leq_T$, and increasing on a cone. Then for any $\alpha<\omega_1$ either\\
(i) for some $\beta\leq\alpha$, $f(x)\equiv_Tx^{(\beta)}$ cofinally or\\
(ii) $(x^{(\alpha)}<_Tf(x))$ cofinally.
\end{theorem}

There are two notable disanalogies between Theorem \ref{zzz} and Theorem \ref{Slaman-Steel}. First, Theorem \ref{zzz} guarantees only that sufficiently constrained functions are \emph{somewhere} equivalent to an iterate of $\mathsf{Con}$, whereas Theorem \ref{Slaman-Steel} guarantees \emph{cofinal} equivalence with an iterate of the Turing jump. Second, by assuming $\mathsf{AD}$, Slaman and Steel inferred that this behavior happens not only cofinally but also \emph{on a cone} in the Turing degrees. There is no obvious analogue of $\mathsf{AD}$ from which one can infer that if cofinally many Lindenbaum degrees have a property then every element in some non-trivial ideal of Lindenbaum degrees has that property.

We then turn our attention to a generalization of consistency, namely, 1-consistency. Recall that a theory $T$ is \emph{1-consistent} if $T$ is consistent with the true $\Pi^0_1$ theory of arithmetic. Just as the $\Pi^0_1$ fragments of natural theories are often proof-theoretically equivalent to iterated consistency statements over a weak base theory, the $\Pi^0_2$ fragments of natural theories are often proof-theoretically equivalent to iterated 1-consistency statements over a weak base theory

Conservativity theorems relating 1-consistency and iterated consistency play an important role in the proof-theoretic analysis of arithmetic theories. For instance, it is a consequence of Beklemshev's \emph{reduction principle} \cite{beklemishev2005reflection} that for any $\Pi^0_1$ $\varphi$, 
$$\mathsf{EA} + 1\mathsf{Con}(\mathsf{EA}) \vdash\varphi \textit{ if and only if } \mathsf{EA} + \{\mathsf{Con}^k(\mathsf{EA}):k<\omega\}\vdash \varphi.$$ 
This fact plays an integral role in Beklemishev's \cite{beklemishev2004provability} consistency proof of $\mathsf{PA}$. We show that this conservativity result is drastically violated \emph{in the limit}. For functions $f$ and $g$, we say that $f$ \emph{majorizes} $g$ if there is a consistent $\varphi$ such that for all $\psi$, if $\psi\vdash\varphi$ then $f(\psi)\vdash g(\psi)$; if in addition $\varphi$ is true then we say that $f$ \emph{majorizes $g$ on a true ideal.}

\begin{proposition}\label{majorizing prop}
For any elementary presentation $\alpha$ of a recursive well-ordering, $1\mathsf{Con}$ majorizes $\mathsf{Con}^\alpha$ on a true ideal.
\end{proposition}

It is tempting to conjecture on the basis of this result that $1\mathsf{Con}$ is the weakest monotonic function majorizing each $\mathsf{Con}^\alpha$ for $\alpha$ a recursive well-ordering. We prove that this is not the case.

\begin{theorem}\label{consequence of majorizing prop}
There are infinitely many monotonic functions $f$ such that for every recursive ordinal $\alpha$, there is an elementary presentation $a$ of $\alpha$ such that $f$ majorizes $\mathsf{Con}^a$ on a true ideal but also $1\mathsf{Con}$ majorizes $f$ on a true ideal.
\end{theorem}

Theorem \ref{aaa} demonstrates that for any monotonic $f$ with a sufficiently constrained range, $f$ must agree cofinally with $\mathsf{Con}$. We would like to strengthen \emph{cofinally} to \emph{on a true ideal}. One strategy for establishing this claim would be to show that every set that is closed under $\mathsf{EA}$ provable equivalence and that contains cofinally many true sentences also contains every sentence in some true ideal. We show that this strategy fails.

\begin{proposition}\label{limitative result}
There is a recursively enumerable set $\mathcal{A}$ that contains arbitrarily strong true sentences and that is closed under $\mathsf{EA}$ provable equivalence but does not contain any true ideals.
\end{proposition}

It is not clear whether Theorem \ref{aaa} can be strengthened in the desired manner.

\section{No monotonic function is strictly between the identity and $\mathsf{Con}$}

In this section we prove that no monotonic function sends every consistent $\varphi$ to a sentence with deductive strength strictly between $\varphi$ and $(\varphi\wedge\mathsf{Con}(\varphi))$. Most of the work is contained in the proof of the following lemma.

\begin{lemma}
\label{bbb}
Let $f$ be a monotonic function such that for all consistent $\varphi$, $f(\varphi)$ strictly implies $\varphi$. Then for every true sentence $\varphi$ there is a true sentence $\theta$ such that $\theta\vdash\varphi$ and $f(\theta)\vdash (\theta\wedge \mathsf{Con}(\theta))$.
\end{lemma}

\begin{proof}
Let $f$ be as in the statement of the theorem. By assumption the following statement is true. 
$$\chi:=\forall \zeta(\mathsf{Con}(\zeta)\rightarrow \mathsf{Con}(\zeta\wedge \neg f(\zeta)))$$ 
Let $\varphi$ be a true sentence. Then the sentence $\psi:=\varphi\wedge\chi$ is true. Let $$\theta:=(\psi\wedge (f(\psi)\rightarrow \mathsf{Con}(\psi))).$$ Note that $\theta\vdash\varphi$.

\begin{claim}
$f(\theta)\vdash (\theta\wedge f(\psi))$.
\end{claim}

Clearly $\theta\vdash\psi$. So $f(\theta)\vdash f(\psi)$ since $f$ is monotonic. Also $f(\theta)\vdash\theta$ by assumption.

\begin{claim}
$(\theta\wedge f(\psi))\vdash (\psi\wedge \mathsf{Con}(\psi))$.
\end{claim}

Immediate from the definition of $\theta$.

\begin{claim}
$(\psi\wedge \mathsf{Con}(\psi))\vdash (\theta\wedge \mathsf{Con}(\theta))$.
\end{claim}

Clearly $(\psi\wedge \mathsf{Con}(\psi))\vdash \theta$. It suffices to show that \[(\psi\wedge \mathsf{Con}(\psi))\vdash \mathsf{Con}(\theta).\] We reason as follows.
\begin{flalign*}
(\psi\wedge\mathsf{Con}(\psi))&\vdash\forall \zeta(\mathsf{Con}(\zeta)\rightarrow \mathsf{Con}(\zeta \wedge \neg f(\zeta))) \textrm{ by choice of $\psi$.}\\
&\vdash \mathsf{Con}(\psi)\rightarrow \mathsf{Con}(\psi \wedge \neg f(\psi)) \textrm{ by instantiation.}\\
& \vdash \mathsf{Con}(\psi\wedge\neg f(\psi)) \textrm{ by logic.} \\
& \vdash \mathsf{Con}(\theta) \textrm{ by the definition of $\theta$.}
\end{flalign*} 
It is immediate from the preceding claims that $f(\theta)\vdash (\theta\wedge \mathsf{Con}(\theta))$.
\end{proof}

A number of results follow immediately from the lemma.

\begin{theorem}[Restatement of Theorem \ref{aaa}]
\label{first}
Let $f$ be monotonic. Suppose that for all consistent $\varphi$,\\
(i) $\varphi\wedge \mathsf{Con}(\varphi)$ implies $f(\varphi)$ and\\
(ii) $f(\varphi)$ strictly implies $\varphi$. \\
Then for every true $\varphi$, there is a true $\psi$ such that $\psi\vdash\varphi$ and $[f(\psi)] = [\psi\wedge \mathsf{Con}(\psi)].$
\end{theorem}

\begin{proof}
By the lemma, for every true $\varphi$ there is a true $\psi$ such that $\psi\vdash\varphi$ and $f(\psi)\vdash (\psi\wedge \mathsf{Con}(\psi))$. Since we are assuming that $(\psi\wedge \mathsf{Con}(\psi))\vdash f(\psi)$, it follows that $[f(\psi)]=[\psi\wedge \mathsf{Con}(\psi)]$.
\end{proof}

We note that this theorem applies to a number of previously studied operators. For instance, the theorem applies to the notion of \emph{cut-free consistency}, i.e., consistency with respect to cut-free proofs. $\mathsf{EA}$ does not prove the cut-elimination theorem, which is equivalent to the totality of super-exponentiation (over $\mathsf{EA}$), and does not prove the equivalence of cut-free consistency and consistency. Another such operator is the Friedman-Rathjen-Weiermann \emph{slow consistency operator} discussed in \textsection{1}. Theorem \ref{first} implies that these operators exhibit the same behavior as the consistency operator ``in the limit.'' Indeed, for any $\varphi$ such that $\varphi$ proves the cut-elimination theorem, $\varphi\wedge\mathsf{Con}(\varphi)$ and $\varphi\wedge\mathsf{Con_{CF}}(\varphi)$ are $\mathsf{EA}$-provably equivalent. Likewise, for any $\varphi$ that proves the totality of $F_{\epsilon_0}$, $\varphi\wedge\mathsf{Con}(\varphi)$ and $\varphi\wedge\mathsf{SlowCon}(\varphi)$ are $\mathsf{EA}$-provably equivalent.

As a corollary of Theorem \ref{first} we note that no monotonic function reliably produces sentences strictly between those produced by the identity and by $\mathsf{Con}$.

\begin{corollary}[Restatement of Corollary \ref{no push down}]
There is no monotonic function $f$ such that for all consistent $\varphi$, \\
(i) $\varphi\wedge\mathsf{Con}(\varphi)$ strictly implies $f(\varphi)$ and \\
(ii) $f(\varphi)$ strictly implies $\varphi$.
\end{corollary}

Shavrukov and Visser \cite{shavrukov2014uniform} studied functions over Lindenbaum algberas and discovered a recursive \emph{extensional uniform density function} $g$ for the Lindenbaum algebra of $\mathsf{EA}$, i.e., (i) for any $\varphi$ and $\psi$ such that $\psi$ strictly implies $\varphi$, $g(\langle\varphi,\psi\rangle)$ is a sentence with deductive strength strictly between $\varphi$ and $\psi$ and (ii) if $\mathsf{EA}\vdash(\varphi\leftrightarrow\psi)$ then, for any $\theta$, $[g(\langle\varphi,\theta\rangle)]=[g(\langle\psi,\theta\rangle)]$ and $[g(\langle\theta,\varphi\rangle)]=[g(\langle\theta,\psi\rangle)]$. They asked whether this result could be strengthened by exhibiting a recursive uniform density function that is monotonic in both its coordinates. As a corollary of our theorem we answer their question negatively.

\begin{corollary}
There is no monotonic uniform density function for the Lindenbaum algebra of $\mathsf{EA}$.
\end{corollary}

\begin{proof}
Suppose there were such a function $g$ over the Lindenbaum algebra of $\mathsf{EA}$. Then given any input of the form $\langle\varphi,(\varphi\wedge\mathsf{Con}(\varphi))\rangle$, $g$ would produce a sentence with deductive strength strictly between $\varphi$ and $(\varphi\wedge \mathsf{Con}(\varphi))$. We then note that $f:\varphi\mapsto g(\langle\varphi,(\varphi\wedge \mathsf{Con}(\varphi))\rangle)$ is monotonic, but that for every consistent $\varphi$, $\varphi\wedge\mathsf{Con}(\varphi)$ strictly implies $f(\varphi)$ and $f(\varphi)$ strictly implies $\varphi$, contradicting the previous theorem.
\end{proof}

Our negative answer to the question raised by Shavrukov and Visser makes use of a $\Pi^0_2$ sentence $\forall\zeta(\mathsf{Con}(\zeta)\rightarrow \mathsf{Con}(\zeta \wedge \neg f(\zeta)))$. Shavrukov and Visser raised the following question in private communication.

\begin{question}
Is there a recursive uniform density function for the lattice of $\Pi^0_1$ sentences over $\mathsf{EA}$ that is monotonic in both its coordinates?
\end{question}

\begin{remark}
\label{fff}
It is clear from the proof of the lemma that any monotonic $f$ meeting the hypotheses of Theorem \ref{first} is not only cofinally equivalent to $\mathsf{Con}$; for every true $\psi$ that implies $$\chi:=\forall \zeta(\mathsf{Con}(\zeta)\rightarrow \mathsf{Con}(\zeta \wedge \neg f(\zeta))),$$ there is a true $\theta$ such that $\theta\vdash\psi$ and $[\psi\wedge\mathsf{Con}(\psi)]=[\theta\wedge\mathsf{Con}(\theta)]=[f(\theta)]$.
\end{remark}

This observation points the way toward a corollary of our theorem; namely that any monotonic function strictly meeting the hypotheses of the theorem must have the same range as $\varphi\mapsto(\varphi\wedge \mathsf{Con}(\varphi))$ in the limit. To prove this, we first prove a version of jump inversion---$\varphi \mapsto (\varphi\wedge\mathsf{Con}(\varphi))$ inversion---for Lindenbaum algebras. This is to say that the range of $\mathsf{Con}$ contains a true ideal in the Lindenbaum algebra. A similar result is established for true $\Pi^0_2$ sentences in \cite{andrews2015structure}.

\begin{proposition}
Suppose $\varphi\vdash \mathsf{Con}(\top)$. Then for some $\psi$, $[\varphi]=[(\psi\wedge \mathsf{Con}(\psi))]$.
\end{proposition}
 
\begin{proof}
Let $\psi:=\mathsf{Con}(\top)\rightarrow\varphi$.

\begin{claim}
$\varphi\vdash(\psi\wedge \mathsf{Con}(\psi))$.
\end{claim}

Trivially, $\varphi\vdash\psi$. Since $\varphi\vdash \mathsf{Con}(\top)$, it follows that from the formalized second incompleteness theorem, i.e., $\mathsf{Con}(\top)\vdash\mathsf{Con}(\neg\mathsf{Con}(\top))$, that $\varphi\vdash \mathsf{Con}(\neg \mathsf{Con}(\top))$. But $\neg\mathsf{Con}(\top)$ is the first disjunct of $\psi$, so $\varphi\vdash \mathsf{Con}(\psi)$.

\begin{claim}
$(\psi\wedge \mathsf{Con}(\psi))\vdash\varphi$.
\end{claim}

Note that $\mathsf{Con}(\psi)\vdash \mathsf{Con}(\top)$. The claim then follows since clearly $(\psi\wedge \mathsf{Con}(\top))\vdash\varphi$.
\end{proof}

\begin{corollary}
Let $f$ be monotonic. Suppose that for all consistent $\varphi$,\\
(i) $\varphi\wedge \mathsf{Con}(\varphi)$ implies $f(\varphi)$ and\\
(ii) $f(\varphi)$ strictly implies $\varphi$.\\
Then the intersection of the ranges of $f$ and $\mathsf{Con}$ in the Lindenbaum algebra contains a true ideal.
\end{corollary}

\begin{proof}
Let $\varphi$ be a sentence such that $\varphi\vdash \mathsf{Con}(\top)$ and 
$$\varphi\vdash \forall \zeta(\mathsf{Con}(\zeta)\rightarrow \mathsf{Con}(\zeta \wedge \neg f(\zeta))).$$ 
Note that both of these sentences are true, and hence $\varphi$ is in an element of a true ideal. By the previous proposition, there is a $\psi$ such that $[\psi\wedge \mathsf{Con}(\psi)]=[\varphi]$. By Remark \ref{fff} there is a $\theta$ such that $[f(\theta)]=[\psi\wedge \mathsf{Con}(\psi)]$, that is, $\varphi$ is in the range of $f$.
\end{proof}

\section{Iterating $\mathsf{Con}$ into the transfinite}

By analogy with Martin's Conjecture, we would like to show that there is a natural well-ordered hierarchy of monotonic functions and that the successor for this well-ordering is induced by $\mathsf{Con}$. Thus, we define the iterates of $\mathsf{Con}$ along elementary presentations of well-orderings.

\begin{definition}
\label{elementary}
By an \emph{elementary presentation} of a recursive well-ordering we mean a pair $(\mathcal{D},<)$ of elementary formulas, such that (i) the relation $<$ well-orders $\mathcal{D}$ in the standard model of arithmetic and (ii) $\mathsf{EA}$ proves that $<$ linearly orders the elements satisfying $\mathcal{D}$, (iii) it is elementarily calculable whether an element represents zero or a successor or a limit and (iv) the elementary formulas defining the set of limit ordinals and the successor relation provably in $\mathsf{EA}$ satisfy their corresponding first order definitions in terms of $<$.
\end{definition}

\begin{definition}
\label{iterates}
Given an elementary presentation $\langle\alpha,<\rangle$ of a recursive well-ordering and a sentence $\varphi$, we use G\"{o}del's fixed point lemma to define sentences $\mathbf{Con}^\star(\varphi,\beta)$ for $\beta<\alpha$ as follows.
$$\mathsf{EA}\vdash \mathbf{Con}^\star(\varphi,\beta)\leftrightarrow \forall\gamma<\beta, \mathsf{Con}(\varphi\wedge\mathbf{Con}^\star(\varphi,\gamma)).$$
We use the notation $\mathsf{Con}^\beta(\varphi)$ for $\mathbf{Con}^\star(\varphi,\beta)$.
\end{definition}

\begin{remark}
Note that, since the following clauses are provable in $\mathsf{EA}$.
\begin{itemize}
\item $\mathsf{Con}^0(\varphi) \leftrightarrow \top$
\item $\mathsf{Con}^{\gamma+1}(\varphi) \leftrightarrow \mathsf{Con}(\varphi\wedge\mathsf{Con}^\gamma(\varphi))$
\item $\mathsf{Con}^{\lambda}(\varphi) \leftrightarrow \forall\gamma<\lambda, \mathsf{Con}^{\gamma}(\varphi)$ for $\lambda$ a limit.
\end{itemize}
\end{remark}

Note that this hierarchy is proper for true $\varphi$ by G\"{o}del's second incompleteness theorem. We need to prove that for transfinite $\alpha$, $\mathsf{Con}^\alpha$ is monotonic over the Lindenbaum algebra of $\mathsf{EA}$. Before proving this claim we recall Schmerl's \cite{schmerl1979fine} technique of \emph{reflexive transfinite induction}. Note that ``$\mathsf{Pr}(\varphi)$'' means that $\varphi$ is provable in $\mathsf{EA}$.

\begin{proposition}
\label{RTI}
(Schmerl) Suppose that $<$ is an elementary linear order and that $\mathsf{EA}\vdash\forall\alpha(\mathsf{Pr}(\forall\beta<\alpha ,A(\beta))\rightarrow A(\alpha))$. Then $\mathsf{EA}\vdash\forall\alpha A(\alpha)$.
\end{proposition}

\begin{proof}
From $\mathsf{EA}\vdash\forall\alpha(\mathsf{Pr}(\forall\beta<\alpha ,A(\beta))\rightarrow A(\alpha))$ we infer
\begin{flalign*}
\mathsf{EA}\vdash \mathsf{Pr}(\forall\alpha A(\alpha))&\rightarrow\forall\alpha\mathsf{Pr}(\forall\beta<\alpha,A(\beta))\\
&\rightarrow\forall\alpha A(\alpha).
\end{flalign*}
L\"{o}b's theorem, i.e., 
$$\textrm{if } \mathsf{EA} \vdash \mathsf{Pr}(\zeta)\rightarrow\zeta \textrm{, then } \mathsf{EA}\vdash \zeta,$$ 
then yields $\mathsf{EA}\vdash\forall\alpha A(\alpha)$.
\end{proof}

\begin{proposition}
\label{Mono}
If $\varphi\vdash\psi$, then $\mathsf{Con}^\alpha(\varphi)\vdash \mathsf{Con}^\alpha(\psi)$.
\end{proposition}

\begin{proof}
Let $\mathcal{A}(\beta)$ denote the claim that $\mathsf{Con}^\beta(\varphi)\vdash \mathsf{Con}^\beta(\psi)$.

We want to prove that $\mathcal{A}(\alpha)$, without placing any restrictions on $\alpha$. We prove the equivalent claim that $\mathsf{EA}\vdash\mathcal{A}(\alpha)$. By Proposition \ref{RTI}, it suffices to show that
$$\mathsf{EA}\vdash\forall\alpha(\mathsf{Pr}(\forall\beta<\alpha,\mathcal{A}(\beta))\rightarrow \mathcal{A}(\alpha)).$$
\textbf{Reason within $\mathsf{EA}$.} Suppose that $\mathsf{Pr}(\forall\beta<\alpha,\mathcal{A}(\beta))$, which is to say that 
$$\mathsf{Pr}(\forall\beta<\alpha,\mathsf{Pr} (\mathsf{Con}^\beta(\varphi) \rightarrow \mathsf{Con}^\beta(\psi))).$$
Since $\mathsf{Con}^\alpha(\varphi)$ contains $\mathsf{EA}$, we infer that
$$\mathsf{Con}^\alpha(\varphi)\vdash\forall\beta<\alpha \mathsf{Pr}(\mathsf{Con}^\beta(\varphi) \rightarrow \mathsf{Con}^\beta(\psi)).$$
Since $\mathsf{Con}^\alpha(\varphi)$ proves that for all $\beta<\alpha$, $\mathsf{EA}\nvdash\neg\mathsf{Con}^\beta(\varphi)$ we infer that
$$\mathsf{Con}^\alpha(\varphi)\vdash \forall\beta<\alpha\mathsf{Con}(\mathsf{Con}^\beta(\psi)).$$
Thus, 
$$\mathsf{Con}^\alpha(\varphi)\vdash \forall\beta<\alpha(\mathsf{Con}^\beta(\psi)).$$
This concludes the proof of the proposition.
\end{proof}

Thus, for each predicate $\mathsf{Con}^\alpha$ the function 
$$\varphi\mapsto(\varphi\wedge \mathsf{Con}^\alpha(\varphi))$$
 is monotonic over the Lindenbaum algebra of $\mathsf{EA}$.
 
In this section we show that the functions given by iterated consistency are minimal with respect to each other. We fix an elementary presentation $\alpha$ of a recursive well-ordering. We assume that $f$ is a monotonic function such that for every consistent $\varphi$, $f(\varphi)$ strictly implies $\varphi\wedge \mathsf{Con}^\beta(\varphi)$ for all $\beta<\alpha.$ We would like to relativize the proof of Lemma \ref{bbb} to $\mathsf{Con}^\beta$. However, the proof of Lemma \ref{bbb} relied on the truth of the principle 
$$\forall\zeta(\mathsf{Con}(\zeta)\rightarrow \mathsf{Con}(\zeta\wedge\neg f(\zeta))).$$ It is not in general clear that $\mathsf{Con}^\alpha(\varphi)$ implies $\mathsf{Con}^\alpha(\varphi\wedge\neg f(\varphi))$. To solve this problem, we define a sequence of true sentences $(\theta_\beta)_{\beta\leq \alpha}$ such that for every sentence $\varphi$, if $\varphi\vdash\theta_\beta$ then $\mathsf{Con}^\beta(\varphi)$ implies $\mathsf{Con}^\beta(\varphi\wedge\neg f(\varphi))$. Thus, we are able to relativize the proof of Lemma \ref{bbb} for $\mathsf{Con}^\beta$ to sentences that imply $\theta_\beta$. 

\begin{definition}\label{sentence sequence}

Given an elementary presentation $\alpha$ of a recursive well-ordering, we use G\"{o}del's fixed point lemma to define sentences $\mathbf{\theta}^\star(\beta)$ for $\beta<\alpha$ as follows.
\begin{flalign*}
\mathsf{EA}\vdash& \mathbf{\theta}^\star(\beta)\leftrightarrow \\
& \forall\gamma<\beta (\mathsf{True}_{\Pi_3}(\theta^\star(\gamma))) \wedge \forall\zeta\Big{(}\big{(}\forall\gamma<\beta\mathsf{Pr}(\zeta\rightarrow\theta^\star(\gamma))\big{)} \rightarrow \big{(}\mathsf{Con}^\beta(\zeta) \rightarrow \mathsf{Con}^\beta(\zeta\wedge\neg f(\zeta))\big{)}\Big{)}.
\end{flalign*}

We use the notation $\theta_\beta$ for $\theta^\star(\beta)$.

\end{definition}

\begin{remark}
Note that every sentence in the sequence $(\theta_\beta)_{\beta\leq\alpha}$ has complexity $\Pi^0_3$. Note moreover that for a successor $\beta+1$, $\theta_{\beta+1}$ is equivalent to 
$$\theta_\beta\wedge\forall\zeta\big{(} \mathsf{Pr}(\zeta\rightarrow \theta_\beta) \rightarrow \big{(}\mathsf{Con}^{\beta+1}(\zeta)\rightarrow \mathsf{Con}^{\beta+1}(\zeta\wedge\neg f(\zeta))\big{)}\big{)}.$$
\end{remark}

\begin{lemma}
Let $f$ be monotonic such that, for all $\varphi$, \\
(i) $\varphi\wedge \mathsf{Con}^\alpha(\varphi)$ implies $f(\varphi)$,\\
(ii) if $[f(\varphi)]\neq[\bot]$, then $f(\varphi)$ strictly implies $\varphi\wedge \mathsf{Con}^\beta(\varphi)$ for all $\beta<\alpha.$ \\
Then for each $\beta\leq \alpha$, the sentence $\theta_\beta$ is true.
\end{lemma}

\begin{proof}
Let $f$ be as in the statement of the lemma. We prove the claim by induction on $\beta\leq \alpha$. The \textbf{base case} $\beta=0$ is trivial.\\

For the \textbf{successor case} we assume that $\beta<\alpha$ and that $\theta_\beta$ is true; we want to show that $\theta_{\beta+1}$ is true. So let $\zeta$ be a sentence such that $\zeta\vdash \theta_\beta$. We want to show that $\mathsf{Con}^{\beta+1}(\zeta)$ implies $\mathsf{Con}^{\beta+1}(\zeta\wedge\neg f(\zeta))$. We prove the contrapositive, that $\neg \mathsf{Con}^{\beta+1}(\zeta\wedge\neg f(\zeta))$ implies $\neg \mathsf{Con}^{\beta+1}(\zeta)$.  So suppose $\neg \mathsf{Con}^{\beta+1}(\zeta\wedge\neg f(\zeta))$, i.e., 
\begin{equation} 
\label{eq:dagger} 
\tag{$\dagger$} 
\zeta\wedge\neg f(\zeta)\vdash\neg \mathsf{Con}^\beta(\zeta\wedge\neg f(\zeta)).
\end{equation} 
We reason as follows.

Since $\zeta\vdash\theta_\beta$, $\zeta\vdash\forall\gamma<\beta, \mathsf{True}_{\Pi_3}(\theta_{\gamma})$. From this we infer
\begin{equation} 
\label{eq:star} 
\tag{$\star$} 
\zeta\vdash \mathsf{Pr}(\zeta\rightarrow \forall\gamma<\beta, \mathsf{True}_{\Pi_3}(\theta_{\gamma}))
\end{equation} 
by $\Sigma^0_1$ completeness. Moreover, since $\zeta\vdash\theta_\beta$,
\begin{flalign*}
\zeta&\vdash \forall\varphi \big{(}\big{(}\forall\gamma<\beta \mathsf{Pr}(\varphi\rightarrow\theta_\gamma)\big{)} \rightarrow \big{(}\mathsf{Con}^\beta(\varphi) \rightarrow \mathsf{Con}^\beta(\varphi\wedge\neg f(\varphi)) \big{)} \big{)} \textrm{ by the definition of $\theta_\beta$.} \\
&\vdash \forall\gamma<\beta \mathsf{Pr} \big{(}\zeta\rightarrow\theta_\gamma\big{)} \rightarrow \big{(}\mathsf{Con}^{\beta}(\zeta) \rightarrow \mathsf{Con}^{\beta}(\zeta\wedge\neg f(\zeta))\big{)} \textrm { by instantiation.} \\
&\vdash \mathsf{Con}^{\beta}(\zeta)\rightarrow \mathsf{Con}^{\beta}(\zeta\wedge\neg f(\zeta)) \textrm{ by \eqref{eq:star}.}\\
\zeta\wedge\neg f(\zeta)&\vdash \neg \mathsf{Con}^\beta(\zeta\wedge\neg f(\zeta))\textrm{ by \eqref{eq:dagger}.}\\
&\vdash \neg \mathsf{Con}^\beta(\zeta)\textrm{ by logic.}\\
\zeta&\vdash \mathsf{Con}^\beta(\zeta)\rightarrow f(\zeta) \textrm{ by logic.}
\end{flalign*}
Thus, $(\zeta\wedge\mathsf{Con}^\beta(\zeta))\vdash f(\zeta)$. Since $f(\varphi)$ always strictly implies $\varphi\wedge\mathsf{Con}^\beta(\varphi)$, we infer that 
$$[\zeta\wedge\mathsf{Con}^\beta(\zeta)]=[\bot].$$ 
This is to say that $\neg \mathsf{Con}^{\beta+1}(\zeta)$.\\

For the \textbf{limit case} we let $\beta$ be a limit ordinal and assume that for every $\gamma<\beta$, $\theta_\gamma$ is true. We want to show that $\theta_\beta$ is true. Let $\zeta$ be a sentence such that for every $\gamma<\beta$, $\zeta\vdash\theta_\gamma$. We want to show that $\mathsf{Con}^\beta(\zeta)$ implies $\mathsf{Con}^\beta(\zeta\wedge\neg f(\zeta))$. So assume that $\mathsf{Con}^\beta(\zeta)$, i.e., for every $\gamma<\beta, \mathsf{Con}^\gamma(\zeta)$. Let $\gamma<\beta$. Since $\beta$ is a limit ordinal, $\gamma+1<\beta$. So by the inductive hypothesis $\theta_{\gamma+1}$ is true. That is, by the definition of $\theta_{\gamma+1}$,
$$\forall\varphi\big{(}\mathsf{Pr}(\varphi\rightarrow\theta_\gamma)\rightarrow (\mathsf{Con}^\gamma(\varphi)\rightarrow \mathsf{Con}^\gamma(\varphi\wedge\neg f(\varphi)))\big{)}.$$ 
By instantiation, we infer that 
$$\mathsf{Pr}(\zeta\rightarrow\theta_\gamma)\rightarrow ( \mathsf{Con}^\gamma(\zeta)\rightarrow \mathsf{Con}^\gamma(\zeta\wedge\neg f(\zeta))).$$ 
Since $\zeta\vdash\theta_\gamma$ and $\mathsf{Con}^\gamma(\zeta)$, this means that $\mathsf{Con}^\gamma(\zeta\wedge\neg f(\zeta))$. Since $\gamma$ was a generic ordinal less than $\beta$, we get that $$\forall\gamma<\beta, \mathsf{Con}^\gamma(\zeta\wedge\neg f(\zeta)),$$ i.e., $\mathsf{Con}^\beta(\zeta)$. This completes the proof of the lemma.
\end{proof}

\begin{theorem}[Restatement of Theorem \ref{generalization of aaa}]
Let $f$ be monotonic. Suppose that for all $\varphi$, \\
(i) $\varphi\wedge \mathsf{Con}^\alpha(\varphi)$ implies $f(\varphi)$,\\
(ii) if $[f(\varphi)]\neq[\bot]$, then $f(\varphi)$ strictly implies $\varphi\wedge \mathsf{Con}^\beta(\varphi)$ for all $\beta<\alpha.$ \\
Then for every true $\chi$, there is a true $\psi$ such that $\psi\vdash\chi$ and $[f(\psi)] = [\psi\wedge \mathsf{Con}^\alpha(\psi)].$
\end{theorem}

\begin{proof}
Let $\chi$ be a true sentence. By the lemma, $\theta_\alpha$ is true. So $$\varphi:=\chi\wedge\theta_\alpha$$ is true. We let $$\psi:=\varphi\wedge(f(\varphi)\rightarrow \mathsf{Con}^\alpha(\varphi)).$$ Note that $\psi\vdash\chi$. We now show that $[\psi\wedge \mathsf{Con}^\alpha(\psi)]=[f(\psi)]$.

\begin{claim}
$f(\psi)\vdash (\psi\wedge f(\varphi))$.
\end{claim}

Since $f$ is monotonic.

\begin{claim}
$(\psi\wedge f(\varphi))\vdash (\varphi\wedge \mathsf{Con}^\alpha(\varphi))$.
\end{claim}

By the definition of $\psi$.

\begin{claim}
$(\varphi\wedge \mathsf{Con}^\alpha(\varphi))\vdash (\psi\wedge \mathsf{Con}^\alpha(\psi))$.
\end{claim}

It is clear from the definition of $\psi$ that $(\varphi\wedge \mathsf{Con}^\alpha(\varphi))\vdash \psi$. So it suffices to show that $(\varphi\wedge \mathsf{Con}^\alpha(\varphi))\vdash \mathsf{Con}^\alpha(\psi)$.
\begin{flalign*}
\varphi\wedge \mathsf{Con}^\alpha(\varphi)&\vdash \forall\zeta\big{(}\big{(}\forall\beta<\alpha \mathsf{Pr}(\zeta \rightarrow \theta_\beta)\big{)} \rightarrow \big{(}\mathsf{Con}^\alpha(\zeta)\rightarrow \mathsf{Con}^\alpha(\zeta\wedge\neg f(\zeta))\big{)}\big{)}\textrm{ by choice of $\varphi$.}\\
&\vdash \forall\beta<\alpha \mathsf{Pr}(\varphi\rightarrow \theta_\beta) \rightarrow \big{(}\mathsf{Con}^\alpha(\varphi)\rightarrow \mathsf{Con}^\alpha(\varphi\wedge\neg f(\varphi))\big{)}\textrm{ by instantiation.}\\
&\vdash \forall\beta<\alpha \mathsf{Pr}(\varphi\rightarrow\theta_\beta)\rightarrow \mathsf{Con}^\alpha(\varphi\wedge\neg f(\varphi))\textrm{ by logic.}
\end{flalign*}
Since $\mathsf{Con}^\alpha(\varphi\wedge\neg f(\varphi))\vdash \mathsf{Con}^\alpha(\psi)$, to prove the desired claim it suffices to show that $$\varphi\wedge \mathsf{Con}^\alpha(\varphi)\vdash \forall\beta<\alpha \mathsf{Pr}(\varphi \rightarrow \theta_\beta).$$ 
We reason as follows.
\begin{flalign*}
\varphi&\vdash \theta_\alpha \textrm{ by choice of $\varphi$.}\\
&\vdash \forall\beta<\alpha(\mathsf{True}_{\Pi_3}\theta_\beta) \textrm{ by definition of $\theta_\alpha$.}\\
&\vdash \mathsf{Pr}(\varphi \rightarrow \forall\beta<\alpha(\mathsf{True}_{\Pi_3}\theta_\beta)) \textrm{ by $\Sigma^0_1$ completeness.}\\
&\vdash \forall\beta<\alpha \mathsf{Pr}(\varphi \rightarrow \mathsf{True}_{\Pi_3} \theta_\beta)\\
&\vdash\forall\beta<\alpha \mathsf{Pr}(\varphi\rightarrow \theta_\beta)
\end{flalign*}
It is immediate from the preceding claims that $f(\psi)\vdash \psi\wedge \mathsf{Con}^\alpha(\psi)$. By assumption, $\psi + \mathsf{Con}^\alpha(\psi) \vdash f(\psi)$, so it follows that $[f(\psi)]=[\psi\wedge \mathsf{Con}^\alpha(\psi)]$.
\end{proof}

\begin{corollary}[Restatement of Corollary \ref{generalization of no push down}]
There is no monotonic $f$ such that for all $\varphi$, if $[\varphi\wedge\mathsf{Con}^\alpha(\varphi)]\neq[\bot]$, then both \\
(i) $\varphi\wedge \mathsf{Con}^\alpha(\varphi)$ strictly implies $f(\varphi)$ and\\
(ii) $f(\varphi)$ strictly implies $\varphi\wedge \mathsf{Con}^\beta(\varphi)$ for all $\beta<\alpha$.
\end{corollary}

\section{Finite iterates of $\mathsf{Con}$ are inevitable}

In this section and the next section we prove that the iterates of $\mathsf{Con}$ are, in a sense, inevitable. First we show that, for every natural number $n$, if a monotonic function $f$ is always bounded by $\mathsf{Con}^n$, then it is somewhere equivalent to $\mathsf{Con}^k$ for some $k\leq n$. In \textsection{5}, we turn to generalizations of this result into the effective transfinite.

\begin{theorem}[Restatement of Theorem \ref{finite version of zzz}]
\label{ttt}
Let $n\in\mathbb{N}$. Let $f$ be a monotonic function such that for every $\varphi$, \\
(i) $\varphi\wedge \mathsf{Con}^n(\varphi)$ implies $f(\varphi)$ and \\
(ii) $f(\varphi)$ implies $\varphi$. \\
Then for some $\varphi$ and some $k\leq n$, $[f(\varphi)]=[\varphi\wedge \mathsf{Con}^k(\varphi)]\neq[\bot].$
\end{theorem}

\begin{proof}
We suppose, towards a contradiction, that there is no $\psi$ and no $k\leq n$ such that $[f(\psi)]=[\psi\wedge\mathsf{Con}^k(\psi)]\neq[\bot]$. We then let $\varphi_1$ be a true statement such that 
\begin{flalign*}
&\varphi_1\vdash\forall\zeta(\mathsf{Con}(\zeta)\rightarrow \mathsf{Con}(\zeta\wedge\neg f(\zeta)))\\
&\varphi_1 \vdash\forall k\forall\zeta\big{(}\mathsf{Con}^{k+1}(\zeta)\rightarrow\neg \mathsf{Pr}\big{(}(\zeta\wedge \mathsf{Con}^k(\zeta))\leftrightarrow f(\zeta)\big{)}\big{)}.
\end{flalign*}
The first condition is that $\varphi_1$ proves that for every consistent $\varphi$, $f(\varphi)$ strictly implies $\varphi$. The second condition is that $\varphi_1$ proves that $f(\zeta)$ never coincides with $\zeta\wedge\mathsf{Con}^k(\zeta)$, unless $[\zeta\wedge\mathsf{Con}^k(\zeta)]=[\bot]$.

We define a sequence of statements, starting with $\varphi_1$, as follows: $$\varphi_{k+1}:=\varphi_k\wedge(f(\varphi_k)\rightarrow \mathsf{Con}^k(\varphi_k)).$$ We will use our assumption to show that, for all $k$, $\varphi_k\wedge \mathsf{Con}^k(\varphi_k)\vdash \mathsf{Con}^{k}(\varphi_{k+1})$. From this we will deduce that $[f(\varphi_{n+1})]=[\varphi_{n+1}\wedge\mathsf{Con}^n(\varphi_{n+1})]\neq[\bot]$, contradicting the assumption that $f$ and $\mathsf{Con}^n$ never coincide. Most of the work is contained in the proof of the following lemma.

\begin{lemma}
\label{finite}
For all $k$, for all $j\geq k$, $\big(\varphi_k\wedge \mathsf{Con}^k(\varphi_k)\big) \vdash \mathsf{Con}^k(\varphi_{j})$.
\end{lemma}

\begin{proof}
We prove the claim by a double induction. The primary induction is on $k$. For the \textbf{base case} $k=1$, we prove the claim by induction on $j$. The \emph{base case} $j=1$ follows trivially. For the \emph{inductive step} we assume that $\big(\varphi_1\wedge \mathsf{Con}(\varphi_1)\big)\vdash \mathsf{Con}(\varphi_{j})$ and show that $\big(\varphi_1\wedge \mathsf{Con}(\varphi_1)\big)\vdash \mathsf{Con}(\varphi_{j+1})$.
\begin{flalign*}
\varphi_1\wedge \mathsf{Con}(\varphi_1)&\vdash\forall\zeta(\mathsf{Con}(\zeta)\rightarrow \mathsf{Con}(\zeta\wedge\neg f(\zeta)))\textrm{ by choice of $\varphi_1.$}\\
&\vdash \mathsf{Con}(\varphi_j)\rightarrow \mathsf{Con}(\varphi_j\wedge\neg f(\varphi_j)) \textrm{ by instantiation.}\\
\varphi_1\wedge \mathsf{Con}(\varphi_1)&\vdash \mathsf{Con}(\varphi_j)\textrm{ by the inductive hypothesis.}\\
&\vdash \mathsf{Con}(\varphi_j\wedge\neg f(\varphi_j)) \textrm{ by logic.}\\
&\vdash \mathsf{Con}(\varphi_{j+1}) \textrm{ by definition of $\varphi_{j+1}$.}
\end{flalign*}
For the \textbf{inductive step} we assume that the claim is true of $k-1$, i.e., 
$$\forall j\geq k-1\big{(}\big{(}\varphi_{k-1}\wedge \mathsf{Con}^{k-1}(\varphi_{k-1})\big{)}\vdash\big{(}\mathsf{Con}^{k-1}(\varphi_{j})\big{)}\big{)}.$$ 
We prove the claim for $k$. Once again, we prove the claim by induction on $j$. The \emph{base case} $j=k$ follows trivially. For the \emph{inductive step} we assume that $\varphi_k\wedge \mathsf{Con}^k(\varphi_k)\vdash \mathsf{Con}^k(\varphi_j)$. We want to prove that $\varphi_k\wedge \mathsf{Con}^k(\varphi_k)\vdash \mathsf{Con}^k(\varphi_{j+1})$.
\begin{flalign*}
\varphi_k\wedge \mathsf{Con}^k(\varphi_k)&\vdash \forall x\forall\zeta\big{(}\mathsf{Con}^{x+1}(\zeta)\rightarrow\neg \mathsf{Pr}\big{(}(\zeta\wedge \mathsf{Con}^x(\zeta))\leftrightarrow f(\zeta)\big{)}\big{)} \textrm{ by choice of $\varphi_1$.}\\
&\vdash \mathsf{Con}^k(\varphi_j)\rightarrow\neg \mathsf{Pr}\big{(}(\varphi_j\wedge \mathsf{Con}^{k-1}(\varphi_j))\leftrightarrow f(\varphi_j)\big{)} \textrm{ by instantiation.}\\
\varphi_k\wedge \mathsf{Con}^k(\varphi_k)&\vdash \mathsf{Con}^k(\varphi_j) \textrm{ by the inner inductive hypothesis.}\\
&\vdash \neg \mathsf{Pr}\big{(}(\varphi_j\wedge \mathsf{Con}^{k-1}(\varphi_j))\leftrightarrow f(\varphi_j)\big{)} \textrm{ by logic.}
\end{flalign*}
Thus, $\varphi_k\wedge \mathsf{Con}^k(\varphi_k)$ proves that one of the following cases holds.
$$(\varphi_j\wedge \mathsf{Con}^{k-1}(\varphi_j))\nvdash f(\varphi_j) $$ 
$$f(\varphi_j)\nvdash(\varphi_j\wedge \mathsf{Con}^{k-1}\varphi_j)$$ 
We now show that $\varphi_k\wedge \mathsf{Con}^k(\varphi_k)$ refutes the second option.

\begin{claim}
$\varphi_k\wedge \mathsf{Con}^k(\varphi_k)\vdash \mathsf{Pr}\big{(}f(\varphi_j)\rightarrow(\varphi_j\wedge \mathsf{Con}^{k-1}\varphi_j)\big{)}$.
\end{claim}

By the outer inductive hypothesis, $\mathsf{EA}$ proves the following conditional: 
$$\theta:=\big{(}(\varphi_{j-1}\wedge \mathsf{Con}^{k-1}(\varphi_{j-1}))\rightarrow(\mathsf{Con}^{k-1}(\varphi_{j}))\big{)}.$$ 
Thus, $f(\varphi_j)$ (which contains $\mathsf{EA}$) also proves $\theta$.  We now show that $f(\varphi_j)\vdash \mathsf{Con}^{k-1}(\varphi_{j})$.
\begin{flalign*}
f(\varphi_j)&\vdash \varphi_j\wedge f(\varphi_{j-1}) \textrm{ since $f$ is monotonic.}\\
&\vdash (\varphi_{j-1}\wedge(f(\varphi_{j-1})\rightarrow \mathsf{Con}^{j-1}(\varphi_{j-1})))\wedge f(\varphi_{j-1}) \textrm{ by the definition of $\varphi_j.$}\\
&\vdash \varphi_{j-1}\wedge \mathsf{Con}^{j-1}(\varphi_{j-1}) \textrm{ by logic.}\\
&\vdash \varphi_{j-1}\wedge \mathsf{Con}^{k-1}(\varphi_{j-1}) \textrm{ since $j\geq k$.}\\
&\vdash \mathsf{Con}^{k-1}(\varphi_{j}) \textrm{ since $f(\varphi_j)$ proves $\theta$.}
\end{flalign*}
By $\Sigma^0_1$ completeness, $(\varphi_k\wedge \mathsf{Con}^k(\varphi_k))\vdash \mathsf{Pr}\big{(}f(\varphi_j) \rightarrow \mathsf{Con}^{k-1}(\varphi_{j})\big{)}$.

\begin{claim}
$(\varphi_k\wedge \mathsf{Con}^k(\varphi_k)) \vdash \mathsf{Con}^k (\varphi_{j+1}).$
\end{claim}

We reason as follows.
\begin{flalign*}
(\varphi_k\wedge \mathsf{Con}^k(\varphi_k))&\vdash \neg\mathsf{Pr}\big((\varphi_j \wedge \mathsf{Con}^{k-1}(\varphi_j))\rightarrow f(\varphi_j)\big) \textrm{ by the previous claim.} \\
&\vdash \mathsf{Con}(\varphi_j\wedge\neg f(\varphi_j)\wedge \mathsf{Con}^{k-1}(\varphi_j)). \\
&\vdash \mathsf{Con}(\varphi_{j+1}\wedge \mathsf{Con}^{k-1}(\varphi_j)) \textrm{ by the definition of $\varphi_{j+1}.$}\\
&\vdash \mathsf{Con}(\varphi_{j+1}\wedge \mathsf{Con}^{k-1}(\varphi_{j+1})) \textrm{ by the outer inductive hypothesis.}\\
&\vdash \mathsf{Con}^k(\varphi_{j+1}) \textrm{ by definition of $\mathsf{Con}^k$.}
\end{flalign*}
This concludes the proof of the lemma.
\end{proof}

As an instance of the lemma, we get that $(\varphi_{n}\wedge \mathsf{Con}^{n}(\varphi_{n})) \vdash \mathsf{Con}^{n}(\varphi_{n+1})$. We reason as follows.
\begin{flalign*}
f(\varphi_{n+1})&\vdash \varphi_{n}\wedge (f(\varphi_{n})\rightarrow \mathsf{Con}^n(\varphi_n))\textrm{ by the definition of $\varphi_{n+1}$.}\\
f(\varphi_{n+1})&\vdash f(\varphi_{n})\textrm{ since $f$ is monotonic.}\\
&\vdash \mathsf{Con}^n(\varphi_n)\textrm{ by logic.}\\
&\vdash \mathsf{Con}^{n}(\varphi_{n+1})\textrm{ by the lemma.}
\end{flalign*}
On the other hand, $\varphi_{n+1}\wedge\mathsf{Con}^{n}(\varphi_{n+1})\vdash f(\varphi_{n+1})$ since $f$ is everywhere bounded by $\mathsf{Con}^n$. Thus, $[f(\varphi_{n+1})]=[\varphi_{n+1}\wedge\mathsf{Con}^{n}(\varphi_{n+1})]$, contradicting the assumption that there is no $\psi$ and no $k\leq n$ such that $[f(\psi)]=[\psi\wedge\mathsf{Con}^k(\psi)]\neq[\bot]$. 
\end{proof}

\section{Transfinite iterates of $\mathsf{Con}$ are inevitable.}

Generalizing the proof of Theorem \ref{ttt} into the transfinite poses the following difficulty. Recall that the proof of Theorem \ref{ttt} makes use of a sequence of sentences starting with $\varphi_0:=\top$ where $$\varphi_{k+1}:=\varphi_k\wedge(f(\varphi_k)\rightarrow \mathsf{Con}^k(\varphi_k)).$$ It is not clear what the $\omega$th sentence in the sequence should be. A natural idea is that for a limit ordinal $\lambda$ the corresponding ``limit sentence'' should quantify over the sentences in the sequence beneath it and express, roughly, 
$$\forall \gamma<\lambda\big{(}\mathsf{True}(\varphi_\gamma) \wedge (\mathsf{True}(f(\varphi_\gamma))\rightarrow \mathsf{Con}^{\gamma}(\varphi_\gamma))\big{)}.$$ 
However, if the sentences in the sequence $(\varphi_\gamma)_{\gamma<\lambda}$ have unbounded syntactic complexity, then we are not guaranteed to have a truth-predicate with which we can quantify over them.

Nevertheless, we show that Theorem \ref{ttt} generalizes into the transfinite given an additional assumption on complexity. Note that $\varphi\mapsto(\varphi\wedge\mathsf{Con}(\varphi))$ can be factored into two functions---the identity and $\varphi\mapsto\mathsf{Con}(\varphi)$---the latter of which always produces a $\Pi^0_1$ sentence. For the rest of this section, we will focus on monotonic functions $\varphi\mapsto\varphi\wedge f(\varphi)$ where $f$ is monotonic and also $f(\varphi)\in\Pi^0_1$ for all $\varphi$.

\begin{definition}
A function $f$ is \emph{$\Pi^0_1$} if $f(\varphi)\in\Pi^0_1$ for all $\varphi$.
\end{definition}

For the next theorem we fix an elementary presentation $\Gamma$ of a recursive well-ordering. In the statement of the theorem and throughout the proof $\alpha$, $\beta$, $\gamma$, $\delta$, etc. are names of ordinals from the notation system $\Gamma$. 

\begin{theorem}[Restatement of Theorem \ref{zzz}]
\label{main}
Let $f$ be a monotonic $\Pi^0_1$ function. Then either \\
(i) for some $\beta\leq\alpha$ and some $\varphi$, $[\varphi\wedge f(\varphi)]=[\varphi\wedge \mathsf{Con}^\beta(\varphi)]\neq[\bot]$ or\\
(ii) for some $\varphi$, $(\varphi\wedge\mathsf{Con}^\alpha(\varphi))\nvdash f(\varphi)$.
\end{theorem}

\begin{proof}
Let $f$ be a monotonic $\Pi^0_1$ function such that for every $\varphi$, 
$$(\varphi\wedge\mathsf{Con}^\alpha(\varphi))\vdash (\varphi\wedge f(\varphi)).$$ 
We assume, for the sake of contradiction, that there is no sentence $\zeta$ and no $\beta\leq\alpha$ such that $[\zeta\wedge\mathsf{Con}^\beta(\zeta)]=[\zeta\wedge f(\zeta)]\neq[\bot]$. We then let $\varphi$ be the conjunction of the following four sentences.
$$\forall\zeta(\mathsf{Con}(\zeta)\rightarrow \mathsf{Con}(\zeta\wedge\neg f(\zeta)))$$
$$\forall \beta\leq\alpha \forall\zeta \big{(}\mathsf{Con}^{\beta}(\zeta)\rightarrow \forall\delta<\beta, \neg \mathsf{Pr}\big{(}(\zeta\wedge \mathsf{Con}^\delta(\zeta))\leftrightarrow (\zeta \wedge f(\zeta))\big{)}\big{)}$$
$$\forall\zeta\forall\eta( \mathsf{Pr}(\zeta\rightarrow\eta) \rightarrow \mathsf{Pr}(f(\zeta)\rightarrow f(\eta)) )$$
$$\forall x \big{(}\mathsf{Pr}(\mathsf{True}_{\Pi^0_2}(x))\rightarrow \mathsf{True}_{\Pi^0_2}(x)\big{)}$$

The first expresses that for every consistent $\varphi$, $f(\varphi)$ strictly implies $\varphi$. The second sentence expresses that if $\beta<\alpha$, then $f(\zeta)$ and $\zeta\wedge\mathsf{Con}^\beta(\zeta)$ never coincide, unless $[\zeta\wedge\mathsf{Con}^\beta(\zeta)]=[\bot]$ . The third sentence expresses the monotonicity of $f$. The fourth sentence expresses the $\Pi^0_2$ soundness of $\mathsf{EA}$. Note that each of these sentences is true, so their conjunction $\varphi$ is also true. Each of the four sentences is $\Pi^0_2$, whence so is $\varphi$. 

We are interested in the following sequence $(\varphi_\beta)_{\beta\leq\Gamma}$. Note that the sentences in the sequence $(\varphi_\beta)_{\beta\leq\Gamma}$ all have complexity $\Pi^0_2$. Note moreover that since $\varphi_1$ is true, so is $\varphi_\beta$ for every $\beta$.
\begin{flalign*}
\varphi_1&:=\varphi.\\
\varphi_\gamma&:= \varphi_1\wedge \forall\delta<\gamma \big{(}\mathsf{True}_{\Pi_1}(f(\varphi_\delta))\rightarrow \mathsf{Con}^{\delta}(\varphi_\delta)\big{)} \textrm{ for $\gamma>1$.}
\end{flalign*}

Formally, we define the sequence $(\varphi_\beta)_{\beta\leq\Gamma}$ by G\"{o}del's fixed point lemma as in Definition \ref{sentence sequence}.

\begin{remark}
\label{linear}
We may assume that the ordinal notation system $\Gamma$ is provably linear in $\mathsf{EA}$. Thus, $\mathsf{EA}\vdash\forall\beta\leq\alpha,\forall\gamma<\beta(\mathsf{True}_{\Pi_2}(\varphi_\beta)\rightarrow\mathsf{True}_{\Pi_2}(\varphi_\gamma))$.
\end{remark}

Our goal is to show that 
$$[\varphi_{\alpha+1}\wedge \mathsf{Con}^\alpha(\varphi_{\alpha+1})]=[\varphi_{\alpha+1}\wedge f(\varphi_{\alpha+1})]$$
contradicting the assumption that $f$ and $\mathsf{Con}^\alpha$ never coincide. The main lemmas needed to prove this result are the following.

\begin{lemma}
\label{ggg} $\mathsf{EA}\vdash \forall \gamma\leq\alpha \mathsf{Pr} \big{(} (\varphi_\gamma\wedge\neg f(\varphi_\gamma)) \rightarrow \varphi_\alpha\big{)}.$
\end{lemma}

\begin{lemma}
\label{sss} 
$\mathsf{EA}\vdash\forall\beta\leq\alpha \forall\gamma\leq\beta \mathsf{Pr} \big{(}(\varphi_\beta \wedge \mathsf{Con}^\gamma(\varphi_\beta) )\rightarrow\mathsf{Con}^\gamma(\varphi_\beta \wedge\neg f(\varphi_\beta))\big{)}.$
\end{lemma}

Lemma \ref{ggg} is needed to derive Lemma \ref{sss}. We now show how we use Lemma \ref{sss} to derive Theorem \ref{main}. As an instance of Lemma \ref{sss}, letting $\alpha=\beta=\gamma$, we infer that
$$\mathsf{EA}\vdash \mathsf{Pr} \big{(}(\varphi_\alpha \wedge \mathsf{Con}^\alpha(\varphi_\alpha)) \rightarrow \mathsf{Con}^\alpha(\varphi_\alpha \wedge\neg f(\varphi_\alpha))\big{)}.$$
From the soundness of $\mathsf{EA}$, we infer that
\begin{equation}
\tag{$\mp$}
\varphi_\alpha + \mathsf{Con}^\alpha(\varphi_\alpha) \vdash \mathsf{Con}^\alpha(\varphi_\alpha\wedge\neg f(\varphi_\alpha)).
\end{equation}
We then reason as follows.
\begin{flalign*}
\varphi_{\alpha+1}&\vdash \varphi_{\alpha}\wedge (f(\varphi_{\alpha})\rightarrow \mathsf{Con}^\alpha(\varphi_\alpha))\textrm{ by the definition of $\varphi_{\alpha+1}$.}\\
 f(\varphi_{\alpha+1})&\vdash f(\varphi_{\alpha})\textrm{ since $f$ is monotonic.}\\
\varphi_{\alpha+1} + f(\varphi_{\alpha+1})&\vdash \varphi_\alpha \wedge \mathsf{Con}^\alpha(\varphi_\alpha)\textrm{ by logic.}\\
&\vdash \mathsf{Con}^{\alpha}(\varphi_{\alpha+1})\textrm{ by $\mp$.}
\end{flalign*}
On the other hand, $\varphi_{\alpha+1} + \mathsf{Con}^{\alpha}(\varphi_{\alpha+1})\vdash f(\varphi_{\alpha+1})$ since $f$ is everywhere bounded by $\mathsf{Con}^\alpha$. Since $\varphi_1$ is true, so too is $\varphi_{\alpha+1}$, whence we infer that 
$$[\varphi_{\alpha+1}\wedge\mathsf{Con}^{\alpha}(\varphi_{\alpha+1})]=[\varphi_{\alpha+1}\wedge f(\varphi_{\alpha+1})]\neq[\bot],$$ 
contradicting the claim that there is no sentence $\zeta$ and no $\beta\leq\alpha$ such that $[\zeta\wedge\mathsf{Con}^\beta(\zeta)]=[\zeta\wedge f(\zeta)]\neq[\bot]$.
\end{proof}

It remains to prove Lemma \ref{ggg} and Lemma \ref{sss}. We devote one subsection to each.

\subsection{Proof of Lemma \ref{ggg}}

In this subsection we prove Lemma \ref{ggg}. First we recall the statement of the lemma.

\begin{lemma}[Restatement of Lemma \ref{ggg}]
$$\mathsf{EA}\vdash \forall \gamma\leq\alpha \big{(} \mathsf{Pr}(\varphi_\gamma\wedge\neg f(\varphi_\gamma)) \rightarrow \varphi_\alpha\big{)}.$$
\end{lemma}

\begin{proof}

We reason in $\mathsf{EA}$. Let $\gamma\leq\alpha$. We assume that 
\begin{equation} 
\label{eq:eta}
\tag{$\eta$} 
\mathsf{True}_{\Pi_2}(\varphi_\gamma)\wedge\neg \mathsf{True}_{\Pi_1}(f(\varphi_\gamma)).
\end{equation}

We we want to derive $\varphi_\alpha$, i.e.
$$\varphi_1\wedge\forall \sigma<\alpha(\mathsf{True}_{\Pi_1}(f(\varphi_\sigma))\rightarrow \mathsf{Con}^{\sigma}(\varphi_\sigma)).$$ 

The first conjunct follows trivially from the assumption that $\mathsf{True}_{\Pi_2}(\varphi_\gamma)$. We now prove the second conjunct of $\varphi_\alpha$ in two parts, first for all $\sigma$ such that $\alpha>\sigma\geq \gamma$ and then for all $\sigma<\gamma$. 

$\alpha>\sigma\geq \gamma:$ From the assumption that $\mathsf{True}_{\Pi^0_2}(\varphi_\gamma)$ we infer that $\varphi_1$, whence we infer that $f$ is monotonic. Thus, for all $\delta\geq\gamma$, $f(\varphi_\delta)\vdash f(\varphi_\gamma)$, i.e., $\mathsf{EA}\vdash \big{(}f(\varphi_\delta)\rightarrow f(\varphi_\gamma)\big{)}$. From $\varphi_1$ we also infer that $\mathsf{EA}$ is $\Pi^0_2$ sound, and so we infer that for all $\delta\geq\gamma$, $\mathsf{True}_{\Pi_1}(f(\varphi_\delta))\rightarrow \mathsf{True}_{\Pi_1}(f(\varphi_\gamma))$. From the assumption that $\neg \mathsf{True}_{\Pi_1}(f(\varphi_\gamma))$ we then infer that for all $\delta\geq\gamma$, $\neg \mathsf{True}_{\Pi_1}(f(\varphi_\delta))$, whence for all $\delta\geq\gamma$, $\mathsf{True}_{\Pi_1}(f(\varphi_\delta))\rightarrow\mathsf{Con}^\delta(\varphi_\delta)$.

$\sigma< \gamma:$ By Remark \ref{linear}, \ref{eq:eta} implies that 
$$\forall \sigma<\gamma (\mathsf{True}_{\Pi_1}\big{(}f(\varphi_{\sigma}))\rightarrow \mathsf{Con}^{\sigma}(\varphi_{\sigma})\big{)}.$$ 

This completes the proof of Lemma \ref{ggg}.
\end{proof}

\subsection{Proof of Lemma \ref{sss}}

In this subsection we prove Lemma \ref{sss}. We recall the statement of Lemma \ref{sss}.

\begin{lemma}[Restatement of Lemma \ref{sss}]
$$\mathsf{EA}\vdash\forall\beta\leq\alpha \forall\gamma\leq\beta \mathsf{Pr} \big{(}\varphi_\beta +\mathsf{Con}^\gamma(\varphi_\beta) \rightarrow\mathsf{Con}^\gamma(\varphi_\beta \wedge\neg f(\varphi_\beta))\big{)}.$$
\end{lemma}

The proof of this lemma is importantly different from the proof of Lemma \ref{finite}. In particular, to push the induction through limit stages we need to know not only that the inductive hypothesis is true but also that it is provable in $\mathsf{EA}$. We resolve this issue by using Schmerl's technique of \emph{reflexive transfinite induction} (see Proposition \ref{RTI}).

In the proof of the lemma, we let $\mathcal{C}(\gamma,\delta)$ abbreviate the claim that 
$$\varphi_\delta+\mathsf{Con}^\gamma(\varphi_\delta)\vdash\mathsf{Con}^\gamma(\varphi_\delta \wedge \neg f(\varphi_\delta)).$$

\begin{proof}
We want to show that 
$$\mathsf{EA}\vdash\forall\beta\leq\alpha(\forall\gamma\leq\beta(\mathcal{C}(\gamma,\beta))).$$
By Proposition \ref{RTI} it suffices to show that 
$$\mathsf{EA}\vdash\forall\alpha(\mathsf{Pr}(\forall\beta\leq\alpha\forall\gamma\leq\beta\mathcal{C}(\gamma,\beta))\rightarrow\forall\gamma\leq\alpha\mathcal{C}(\gamma,\alpha)).\footnote{The reader might expect that we need to write ``$\beta<\alpha$'' instead of ``$\beta\leq\alpha$'' in the antecedent for this to match the statement of Proposition \ref{RTI}. However, it is clear from the proof of Proposition \ref{RTI} that this suffices.}$$
Thus, we \textbf{reason in $\mathsf{EA}$} and fix $\alpha$. We assume that
\begin{equation} 
\label{eq:bigtriangleup} 
\tag{$\bigtriangleup$}
\mathsf{Pr}(\forall\beta\leq\alpha, \forall\gamma\leq\beta, \mathcal{C}(\gamma,\beta)).
\end{equation}
We let $\gamma\leq\alpha$ and we want to show that $\mathcal{C}(\gamma,\alpha)$.

Since $\varphi_\alpha\vdash\varphi$ we infer that 
\begin{equation} 
\label{eq:sharp} 
\tag{$\sharp$}
\varphi_\alpha+\mathsf{Con}^\gamma(\varphi_\alpha)\vdash\forall\delta<\gamma, \neg \mathsf{Pr}\big{(}(\varphi_{\alpha}\wedge\mathsf{Con}^\delta(\varphi_\alpha)) \leftrightarrow (\varphi_\alpha\wedge f(\varphi_\alpha))\big{)}.
\end{equation}

We first note that both
\begin{flalign*}
\varphi_\alpha&\vdash \forall\delta<\gamma (\mathsf{True}_{\Pi_1}(f(\varphi_{\delta}))\rightarrow \mathsf{Con}^{\delta}(\varphi_{\delta})) \textrm{ by the definition of $\varphi_\alpha$ and also}\\
\varphi_\alpha + f(\varphi_\alpha)&\vdash\forall\delta<\gamma \mathsf{Pr}(f(\varphi_\alpha) \rightarrow f(\varphi_\delta)) \textrm{ since $\varphi_1$ proves the monotonicity of $f$.}\\
& \vdash \forall\delta<\gamma (f(\varphi_\alpha)\rightarrow \mathsf{True}_{\Pi_1}(f(\varphi_\delta))) \textrm{ since $\varphi_1$ proves the $\Pi^0_2$ soundness of $\mathsf{EA}$.}\\
&\vdash \forall\delta<\gamma, \mathsf{True}_{\Pi_1}(f(\varphi_{\delta})) \textrm{ by logic.}
\end{flalign*}
Thus, we may reason as follows.
\begin{flalign*}
\varphi_\alpha+f(\varphi_\alpha)&\vdash \forall\delta<\gamma, \mathsf{Con}^\delta(\varphi_\delta)\\
&\vdash \forall\delta<\gamma, \mathsf{Con}^{\delta}(\varphi_{\delta}\wedge\neg f(\varphi_\delta))) \textrm{ since \eqref{eq:bigtriangleup} delivers $\mathcal{C}(\delta,\delta)$.}\\
&\vdash \forall\delta<\gamma, \mathsf{Con}^{\delta}(\varphi_\alpha) \textrm{ by Lemma \ref{ggg}.}
\end{flalign*}
Thus, by $\Sigma^0_1$ completeness,
$$\mathsf{EA}\vdash\forall\delta<\gamma \mathsf{Pr} \big{(} (\varphi_\alpha\wedge f(\varphi_\alpha)) \rightarrow \mathsf{Con}^\delta(\varphi_\alpha)\big{)}.$$
Combined with \eqref{eq:sharp}, this delivers
\begin{flalign*}
\varphi_\alpha+\mathsf{Con}^\gamma(\varphi_\alpha)&\vdash\forall\delta<\gamma \neg \mathsf{Pr} \big{(} (\varphi_{\alpha} \wedge \mathsf{Con}^\delta(\varphi_\alpha)) \rightarrow f(\varphi_\alpha)\big{)}.\\
&\vdash\forall\delta<\gamma, \mathsf{Con}(\varphi_{\alpha} \wedge \neg f(\varphi_\alpha) \wedge\mathsf{Con}^\delta(\varphi_\alpha)).\\
&\vdash\forall\delta<\gamma, \mathsf{Con}\big{(}\varphi_{\alpha} \wedge \neg f(\varphi_\alpha) \wedge\mathsf{Con}^\delta(\varphi_\alpha\wedge\neg f(\varphi_\alpha))\big{)} \textrm{ since \eqref{eq:bigtriangleup} delivers $\mathcal{C}(\delta,\alpha)$.}\\
&\vdash \mathsf{Con}^\gamma(\varphi_\alpha\wedge\neg f(\varphi_\alpha)).
\end{flalign*}
This completes the proof of Lemma \ref{sss}.
\end{proof}

Theorem \ref{main} shows the inevitability of the consistency operator. For a sufficiently constrained monotonic function $f$, $f$ must coincide with an iterate of $\mathsf{Con}$ on some non-trivial sentence. However, it is not clear from the proofs of Theorem \ref{ttt} or Theorem \ref{main} that $f$ must coincide with $\mathsf{Con}$ on a \emph{true} sentence.

\begin{question}
Let $f$ be a monotonic $\Pi^0_1$ function. Suppose that for every $\varphi$, $$(\varphi\wedge\mathsf{Con}^\alpha(\varphi))\vdash f(\varphi).$$ Must there be some $\beta\leq\alpha$ and some \textbf{true} $\varphi$ such that $$[\varphi\wedge f(\varphi)]=[\varphi\wedge \mathsf{Con}^\beta(\varphi)]?$$
\end{question}

\section{1-consistency and iterated consistency}

Just as the $\Pi^0_1$ fragments of natural theories can often be approximated by iterated consistency statements, the $\Pi^0_2$ fragments of natural theories can often be approximated by iterated $1$-consistency statements. A theory $T$ is \emph{1-consistent} if $T+\mathrm{Th}_{\Pi^0_1}(\mathbb{N})$ is consistent. The $1$-consistency of $\mathsf{EA}+\varphi$ can be expressed by the following $\Pi^0_{2}$ sentence, $1\mathsf{Con}(\varphi)$: $$\forall x (\mathsf{True}_{\Pi^0_1}(x)\rightarrow \mathsf{Con}(\varphi\wedge\mathsf{True}_{\Pi^0_1}(x))).$$

In this section, we investigate the relationship between $1$-consistency and iterated consistency. First, we show that $1\mathsf{Con}$ majorizes every iterate of $\mathsf{Con}^\alpha$.

\begin{proposition}[Restatement of Proposition \ref{majorizing prop}]
\label{xxx}
For any elementary presentation $\alpha$ of a recursive well ordering, there is a true sentence $\varphi$ such that for every $\psi$, if $\psi\vdash\varphi$, then $(\psi\wedge 1\mathsf{Con}(\psi))$ implies $(\psi\wedge \mathsf{Con}^\alpha(\psi))$. Moreover, if $[\psi\wedge \mathsf{Con}^\alpha(\psi)]\neq[\bot]$ then $(\psi\wedge 1\mathsf{Con}(\psi))$ strictly implies $(\psi\wedge \mathsf{Con}^\alpha(\psi))$.
\end{proposition}

\begin{proof} 

Let $\alpha$ be an elementary presentation of a recursive well-ordering. Let $\varphi$ be a true sentence such that $\varphi\vdash \mathsf{TI}^{\alpha}_{\Pi^0_1}$, i.e., $\varphi$ implies the validity of transfinite induction along $\alpha$ for ${\Pi^0_1}$ predicates. We prove that 
$$(\varphi\wedge 1\mathsf{Con}(\varphi))\vdash \mathsf{Con}^{\alpha+1}(\varphi).$$

Since $\varphi\wedge 1\mathsf{Con}(\varphi)\vdash  \mathsf{TI}^{\alpha}_{\Pi^0_1}$, it suffices to show that:\\
\textbf{Base case:} $(\varphi\wedge 1\mathsf{Con}(\varphi))\vdash \mathsf{Con}(\varphi)$\\
\textbf{Successor case:} $(\varphi\wedge 1\mathsf{Con}(\varphi))\vdash \forall\beta<\alpha( \mathsf{Con}^\beta(\varphi)\rightarrow\mathsf{Con}^{\beta+1}(\varphi))$\\
\textbf{Limit case:} $(\varphi\wedge 1\mathsf{Con}(\varphi))\vdash \forall\lambda\Big{(}lim(\lambda)\rightarrow \big{(}(\forall\beta<\lambda\mathsf{Con}^\beta(\varphi))\rightarrow\mathsf{Con}^\lambda(\varphi)\big{)}\Big{)}$

The \textbf{base case} and the \textbf{limit case} are both trivial. For the \textbf{successor case} we first note that by the definition of $1\mathsf{Con}(\varphi)$,
$$1\mathsf{Con}(\varphi)\vdash \forall x (\mathsf{True}_{\Pi^0_1}(x)\rightarrow \mathsf{Con}(\varphi\wedge\mathsf{True}_{\Pi^0_1}(x))),$$
and so by substituting $\mathsf{Con}^\beta(\varphi)$ in for $x$,
\begin{equation}
 \label{eq:oplus} 
 \tag{$\oplus$}
1\mathsf{Con}(\varphi)\vdash \mathsf{True}_{\Pi^0_1}(\mathsf{Con}^\beta(\varphi))\rightarrow \mathsf{Con}(\varphi\wedge\mathsf{True}_{\Pi^0_1}(\mathsf{Con}^\beta(\varphi))).
\end{equation}
Thus, we reason as follows.
\begin{flalign*}
1\mathsf{Con}(\varphi)\vdash\mathsf{Con}^\beta(\varphi)&\rightarrow \mathsf{Con}(\varphi\wedge\mathsf{True}_{\Pi^0_1}(\mathsf{Con}^\beta(\varphi))) \textrm{ by \eqref{eq:oplus}.}\\
&\rightarrow \mathsf{Con}(\varphi\wedge \mathsf{Con}^\beta(\varphi)).\\
&\rightarrow \mathsf{Con}^{\beta+1}(\varphi) \textrm{ by the definition of $\mathsf{Con}^{\beta+1}$.}
\end{flalign*}
It is clear that the implication $\varphi\wedge\mathsf{1Con}(\varphi)\vdash \varphi\wedge \mathsf{Con}^\alpha(\varphi)$ is strict as long as $[\varphi\wedge \mathsf{Con}^\alpha(\varphi)]\neq[\bot]$. This completes the proof of the proposition.
\end{proof}

In light of the previous proposition, one might conjecture that $1\mathsf{Con}$ is the weakest monotonic function majorizing every function of the form $\mathsf{Con}^\alpha$ for some recursive well-ordering $\alpha$ on true sentences. However, this is not so. To demonstrate this, we use a recursive linear order that has no hyperarithmetic infinite descending sequences. Harrison \cite{harrison1968recursive} introduced such an ordering with order-type $\omega_1^{CK}\times(1+\mathbb{Q})$; see also Feferman and Spector \cite{feferman1962incompleteness} who consider such orderings in the context of iterated reflection principles. We use a presentation $\mathcal{H}$ of Harrison's ordering such satisfying the conditions explicated in Definition \ref{elementary}. We note that since $\mathcal{H}$ has no hyperarithmetic descending sequences, transfinite induction along $\mathcal{H}$ for $\Pi^0_1$ properties is valid. Our idea is to produce a function stronger than each $\mathsf{Con}^\alpha$ but weaker than $1\mathsf{Con}$ by iterating $\mathsf{Con}$ along the Harrison linear order. 

\begin{theorem}[Restatement of Theorem \ref{consequence of majorizing prop}]
There are infinitely many monotonic functions $f$ such that for every recursive ordinal $\alpha$, there is an elementary presentation $a$ of $\alpha$ such that $f$ majorizes $\mathsf{Con}^a$ on a true ideal but also $1\mathsf{Con}$ majorizes $f$ on a true ideal.
\end{theorem}

\begin{proof}
In Definition \ref{iterates}, we used G\"{o}del's fixed point lemma to produce iterates of $\mathsf{Con}$ along an elementary well-ordering. We similarly use G\"{o}del's fixed point lemma to define sentences $\mathbf{Con}^\star(\varphi,\beta)$ for $\beta\in\mathcal{H}$ as follows.
$$\mathsf{EA}\vdash \mathbf{Con}^\star(\varphi,\beta)\leftrightarrow \forall\gamma<_\mathcal{H}\beta, \mathsf{Con}(\varphi\wedge\mathbf{Con}^\star(\varphi,\gamma)).$$
We use the notation $\mathsf{Con}^\beta(\varphi)$ for $\mathbf{Con}^\star(\varphi,\beta)$.
Recall that we are assuming that it is elementarily calculable whether an element of $\mathcal{H}$ is zero or a successor or a limit. Thus, the following clauses are provable in $\mathsf{EA}$.
\begin{itemize}
\item $\mathsf{Con}^0(\varphi) \leftrightarrow \top$
\item $\mathsf{Con}^{\gamma+1}(\varphi) \leftrightarrow \mathsf{Con}(\varphi\wedge\mathsf{Con}^\gamma(\varphi))$
\item $\mathsf{Con}^{\lambda}(\varphi) \leftrightarrow \forall\gamma<_\mathcal{H}\lambda, \mathsf{Con}^{\gamma}(\varphi)$ for $\lambda$ a limit.
\end{itemize}

\begin{claim}
For $\gamma\in \mathcal{H}$, the function $\varphi\mapsto {Con}^\gamma(\varphi)$ is monotonic.
\end{claim}

This follows immediately from Proposition \ref{Mono}. Note that in the statement of Lemma \ref{RTI} we assume only that $<$ is an elementary \emph{linear} ordering, not a well-ordering. 

\begin{claim}
There are infinitely many monotonic functions $f$ such that for every recursive well-ordering $\alpha$, there is an elementary presentation $a$ of $\alpha$ such that $f$ majorizes $\mathsf{Con}^a$ on true sentences.
\end{claim}

If $x<_\mathcal{H}y$ then $\mathsf{Con}^y(\varphi)$ strictly implies $\mathsf{Con}^x(\varphi)$ for every $\varphi$ such that $\mathsf{Con}^x(\varphi)\neq[\bot]$. Given the order type of $\mathcal{H}$, this means that for infinitely many $\gamma$, for every recursive well-ordering $\alpha$, $\mathsf{Con}^\gamma$ majorizes $\mathsf{Con}^a$ where $a$ represents $\alpha$ in $\mathcal{H}$.

\begin{claim}
$1\mathsf{Con}$ majorizes $\mathsf{Con}^a$ on true sentences for each $a\in\mathcal{H}$.
\end{claim}
Since every $\Pi^0_1$ definable subset of $\omega$ has an $\mathcal{H}$-least element, the sentence $\mathsf{TI}^{\mathcal{H}}_{\Pi^0_1}$, which expresses the validity of transfinite induction along $\mathcal{H}$ for ${\Pi^0_1}$ predicates, is true. But then if $\varphi\vdash \mathsf{TI}^{\mathcal{H}}_{\Pi^0_1}$, then for any $\gamma\in\mathcal{H}$, $(\varphi\wedge 1\mathsf{Con}(\varphi))$ strictly implies $(\varphi\wedge\mathsf{Con}^\gamma(\varphi))$ as long as $[(\varphi\wedge\mathsf{Con}^\gamma(\varphi))]\neq[\bot]$, as in Proposition \ref{xxx}.
\end{proof}

\section{An unbounded recursively enumerable set that contains no true ideals}

In this section we prove a limitative result. Theorem \ref{first} demonstrates that if $f$ is monotonic and that for all consistent $\varphi$, (i) $\varphi\wedge \mathsf{Con}(\varphi)$ implies $f(\varphi)$ and (ii) $f(\varphi)$ strictly implies $\varphi$, then for cofinally many true $\varphi$, $[f(\varphi)] = [\varphi\wedge \mathsf{Con}(\varphi)].$ It is natural to conjecture that cofinal equivalence with $\mathsf{Con}$ be strengthened to equivalence to $\mathsf{Con}$ \emph{in the limit}, i.e., on a true ideal. One strategy to strengthen Theorem \ref{first} in this way would be to show that every recursively enumerable set that contains arbitrarily strong true sentences and that is closed under provable equivalence contains a true ideal.

We now show that the aforementioned strategy fails. To this end, we define a recursively enumerable set $\mathcal{A}$ that contains arbitrarily strong true sentences and that is closed under provable equivalence but does not contain any true ideals. We are grateful to Matthew Harrison-Trainor for simplifying the proof of the following proposition.

\begin{proposition}[Restatement of Proposition \ref{limitative result}]
\label{eee}
There is a recursively enumerable set $\mathcal{A}$ that contains arbitrarily strong true sentences and that is closed under $\mathsf{EA}$ provable equivalence but does not contain any true ideals.
\end{proposition}

\begin{proof}
Let $\{\varphi_0,\varphi_1,...\}$ be an effective G\"{o}del numbering of the language of arithmetic. We describe the construction of $\mathcal{A}$ in stages. During a stage $n$ we may \emph{activate} a sentence $\psi$, in which case we say that $\psi$ is \emph{active} until it is \emph{deactivated} at some later stage $n+k$. After describing the construction of $\mathcal{A}$ we verify that $A$ has the desired properties.

\textbf{Stage 0:} Numerate $\varphi_0$ and $\neg\varphi_0$ into $\mathcal{A}$. Activate the sentences $(\varphi_0\wedge \mathsf{Con}(\varphi_0))$ and $(\neg\varphi_0\wedge \mathsf{Con}(\neg\varphi_0))$.

\textbf{Stage n+1:} There are finitely many active sentences. For each such sentence $\psi$, numerate $\theta_0:=(\psi\wedge\varphi_{n+1})$ and $\theta_1:=(\psi\wedge\neg\varphi_{n+1})$ into $\mathcal{A}$. Deactivate the sentence $\psi$ and activate the sentences $(\theta_0\wedge \mathsf{Con}(\theta_0))$ and $(\theta_1\wedge \mathsf{Con}(\theta_1))$.

We dovetail the construction with a search through $\mathsf{EA}$ proofs. If we ever see that $\mathsf{EA}\vdash\varphi\leftrightarrow\psi$ for some $\varphi$ that we have already numerated into $\mathcal{A}$, then we numerate $\psi$ into $\mathcal{A}$.

Now we check that $\mathcal{A}$ has the desired properties. It is clear that $\mathcal{A}$ is recursively enumerable and that $\mathcal{A}$ is closed under $\mathsf{EA}$ provable equivalence.

\begin{claim}
$\mathcal{A}$ contains arbitrarily strong true sentences. That is, for each true sentence $\varphi$, there is a true sentence $\psi$ such that $\psi\vdash\varphi$ and $\psi\in \mathcal{A}$.
\end{claim}

At any stage in the construction of $\mathcal{A}$, there are finitely many active sentences, $\psi_0$, ..., $\psi_k$. An easy induction shows that exactly one of $\psi_0,...,\psi_k$ is true. Indeed, exactly one of $\varphi_0$ or $\neg\varphi_0$ is true, and hence so is exactly one of $\varphi_0\wedge \mathsf{Con}(\varphi_0)$ and $\neg\varphi_0\wedge \mathsf{Con}(\neg\varphi_0)$. And if $\theta$ is true, then so is exactly one of $\zeta_0:=\theta\wedge\varphi_k$ and $\zeta_1:=\theta\wedge\neg\varphi_k$, and hence so too is exactly one of $\zeta_0\wedge \mathsf{Con}(\zeta_0)$ and $\zeta_1\wedge \mathsf{Con}(\zeta_1)$.

Let $\varphi_k$ be a true sentence. At stage $k$ in the construction of $\mathcal{A}$ there are only finitely many active sentences $\psi_0,...,\psi_n$. We have already seen that exactly one of $\psi_i$ is true. But then $\varphi_k\wedge\psi_i$ is true, $(\varphi_k\wedge\psi_i\vdash\varphi_k)$, and $(\varphi_k\wedge\psi_i)$ is numerated into $\mathcal{A}$.

\begin{claim}
$\mathcal{A}$ contains no true ideals.
\end{claim}

An easy induction shows that if $\psi_0$ and $\psi_1$ are both active at the same stage, then for any $\theta$, if $\theta$ implies both $\psi_0$ and $\psi_1$ then $\theta\in[\bot]$.

Let $\varphi$ be a true sentence in $\mathcal{A}$. By the previous remark, the only sentences in $\mathcal{A}$ that strictly imply $\varphi$ are (i) $\mathsf{EA}$ refutable sentences and (ii) sentences that imply $\varphi\wedge \mathsf{Con}(\varphi)$. Since the Lindenbaum algebra of $\mathsf{EA}$ is dense, this means there is some $\psi$ such that $(\varphi\wedge \mathsf{Con}(\varphi))$ strictly implies $\psi$ strictly implies $\varphi$ but $\psi\notin \mathcal{A}$. 
\end{proof}

The following questions remain.

\begin{question}
Is the relation of cofinal agreement on true sentences an equivalence relation on recursive monotonic operators?
\end{question}

\begin{question}
Let $f$ be recursive and monotonic. Suppose that for all consistent $\varphi$,\\
(i) $\varphi\wedge \mathsf{Con}(\varphi)$ implies $f(\varphi)$ and\\
(ii) $f(\varphi)$ implies $\varphi$. \\
Must $f$ be equivalent to the identity or to $\mathsf{Con}$ on a true ideal?
\end{question}

  \bibliographystyle{plain}
  \bibliography{consistency.bib}

\begin{thebibliography}{10}

\bibitem{andrews2015structure}
Uri Andrews, Mingzhong Cai, David Diamondstone, Steffen Lempp, and Joseph~S
  Miller.
\newblock On the structure of the degrees of relative provability.
\newblock {\em Israel Journal of Mathematics}, 207(1):449--478, 2015.

\bibitem{beklemishev1991provability}
Lev~D Beklemishev.
\newblock Provability logics for natural turing progressions of arithmetical
  theories.
\newblock {\em Studia Logica}, 50(1):107--128, 1991.

\bibitem{beklemishev1995iterated}
Lev~D Beklemishev.
\newblock Iterated local reflection versus iterated consistency.
\newblock {\em Annals of Pure and Applied Logic}, 75(1-2):25--48, 1995.

\bibitem{beklemishev2003proof}
Lev~D Beklemishev.
\newblock Proof-theoretic analysis by iterated reflection.
\newblock {\em Archive for Mathematical Logic}, 42(6):515--552, 2003.

\bibitem{beklemishev2004provability}
Lev~D Beklemishev.
\newblock {Provability algebras and proof-theoretic ordinals, I}.
\newblock {\em Annals of Pure and Applied Logic}, 128(1-3):103--123, 2004.

\bibitem{beklemishev2005reflection}
Lev~D Beklemishev.
\newblock Reflection principles and provability algebras in formal arithmetic.
\newblock {\em Russian Mathematical Surveys}, 60(2):197, 2005.

\bibitem{feferman1962incompleteness}
Solomon Feferman and Clifford Spector.
\newblock Incompleteness along paths in progressions of theories.
\newblock {\em The Journal of Symbolic Logic}, 27(4):383--390, 1962.

\bibitem{friedman2013slow}
Sy-David Friedman, Michael Rathjen, and Andreas Weiermann.
\newblock Slow consistency.
\newblock {\em Annals of Pure and Applied Logic}, 164(3):382--393, 2013.

\bibitem{harrison1968recursive}
Joseph Harrison.
\newblock Recursive pseudo-well-orderings.
\newblock {\em Transactions of the American Mathematical Society},
  131(2):526--543, 1968.

\bibitem{joosten2016turing}
Joost~J Joosten.
\newblock {Turing--Taylor expansions for arithmetic theories}.
\newblock {\em Studia Logica}, 104(6):1225--1243, 2016.

\bibitem{pour1967deduction}
Marian~Boykan Pour-El and Saul Kripke.
\newblock Deduction-preserving ``recursive isomorphisms'' between theories.
\newblock {\em Fundamenta Mathematicae}, 61:141--163, 1967.

\bibitem{schmerl1979fine}
Ulf~R Schmerl.
\newblock A fine structure generated by reflection formulas over primitive
  recursive arithmetic.
\newblock {\em Studies in Logic and the Foundations of Mathematics},
  97:335--350, 1979.

\bibitem{shavrukov2014uniform}
V~Yu Shavrukov and Albert Visser.
\newblock {Uniform density in Lindenbaum algebras}.
\newblock {\em Notre Dame Journal of Formal Logic}, 55(4):569--582, 2014.

\bibitem{slaman1988definable}
Theodore~A Slaman and John~R Steel.
\newblock Definable functions on degrees.
\newblock In {\em Cabal Seminar 81--85}, pages 37--55. Springer, 1988.

\end{thebibliography}

\end{document}